\documentclass[a4paper, 11pt]{article}
\usepackage[mathscr]{eucal}
\usepackage{amssymb}
\usepackage{latexsym}
\usepackage{amsthm}
\usepackage{amsmath}
\usepackage{psfrag}
\usepackage{a4wide}

\newcommand{\x}{\boldsymbol{x}}
\newcommand{\y}{\boldsymbol{y}}

\usepackage{tikz}
\newcommand{\Z}{\mathbb{Z}}
\newcommand{\Q}{\mathbb{Q}}

\newtheorem{theorem}{Theorem}[section]

\newtheorem{lemma}[theorem]{Lemma}

\theoremstyle{definition}

\newtheorem{conjecture}[theorem]{Conjecture}

\makeatletter
\@addtoreset{equation}{section}

\makeatother

\title{On the existence and non-existence of spherical $m$-stiff configurations}

\author{ 
Eiichi Bannai\thanks{Faculty of Mathematics, Kyushu University (emeritus), Japan. 
	{\tt bannai@math.kyushu-u.ac.jp}}
\and
Hirotake Kurihara\thanks{Department of Applied Science, Yamaguchi University, 2-16-1 Tokiwadai, Ube 755-8611, Japan. 
{\tt kurihara-hiro@yamaguchi-u.ac.jp}}
\and
Hiroshi Nozaki\thanks{Department of Mathematics Education, 
	Aichi University of Education, 
	1 Hirosawa, Igaya-cho, 
	Kariya, Aichi 448-8542, 
	Japan. {\tt hnozaki@auecc.aichi-edu.ac.jp}}
}

\begin{document}

\maketitle

\renewcommand{\thefootnote}{\fnsymbol{footnote}}


\begin{abstract}
This paper investigates the existence of $m$-stiff configurations in the unit sphere $S^{d-1}$, which are spherical $(2m-1)$-designs that lie on $m$ parallel hyperplanes. We establish two non-existence results: (1) for each fixed integer $m > 5$, there exists no $m$-stiff configuration in $S^{d-1}$ for sufficiently large $d$; (2) for each fixed integer $d > 10$, there exists no $m$-stiff configuration in $S^{d-1}$ for sufficiently large  $m$. 
Furthermore, we provide a complete classification of the dimensions where $m$-stiff configurations exist for $m=2,3,4,5$. 
We also determine the non-existence (and the existence) of $m$-stiff configurations in $S^{d-1}$ for small $d$ ($3 \leq d \leq 120$) with arbitrary $m$, and also for small $m$ ($6 \leq m \leq 10$) with arbitrary $d$. 
Finally, we conjecture that there is no $m$-stiff configuration in $S^{d-1}$ for $(d,m)$ with $d\geq 3$ and $m\geq 6$. 
\end{abstract}
\bigskip


\section{Introduction}
Optimal configurations for potential energy problems have been widely studied and are of particular interest in discrete geometry; for example, \cite{B24,BDHSSarXiv,CK07,S75}. In particular, $m$-stiff configurations are worth studying due to their structured properties. 
For positive integers $m,d$, a finite subset $X$ of the $(d-1)$-dimensional unit sphere $S^{d-1}$ is called an {\it $m$-stiff} if it is a spherical $(2m-1)$-design (whose definition will be given below) and satisfies the condition that the set  
\[
\mathcal{D}_m(X) = \{z \in S^{d-1} \colon\, |D(z,X)| \leq m \}
\]  
is non-empty, where $D(z,X) = \{\langle z, x \rangle \colon\, x \in X\}$ and $\langle \cdot, \cdot \rangle$ denotes the standard inner product in $\mathbb{R}^d$. 
It seems that the name $m$-stiff was first used by Borodachov \cite{B24}. 
It has been shown that, under some differentiability assumptions on $f:[0,4] \to (-\infty, \infty)$, an $m$-stiff configuration minimizes the $f$-potential  
\[
\sum_{x \in X} f(|z-x|^2)
\]  
at every point $z$ of $\mathcal{D}_m(X)$ \cite[Theorem 4.3]{B22}.

Here, a {\it spherical $t$-design} is a finite subset $X$ of $S^{d-1}$ satisfying  
\[
\frac{1}{|X|}\sum_{x \in X} f(x)=\frac{1}{|S^{d-1}|}\int_{x \in S^{d-1}} f(x) dx
\]
for any polynomial $f$ of degree at most $t$ \cite{DGS77}, where $|S^{d-1}|$ is the volume of $S^{d-1}$. 
For a spherical $t$-design $X$, a natural lower bound is known \cite{DGS77}:  
\[
|X|\geq 
\begin{cases}
\binom{d+m-1}{m}+\binom{d+m-2}{m-1} &\text{ if $t=2m$}, \\
2\binom{d+m-2}{m-1} &\text{ if $t=2m-1$}. 
\end{cases}
\]
A spherical $t$-design is said to be {\it tight} if it attains this bound. 
Tight spherical $t$-designs are rare, and they do not exist for $t > 11$ and $d>2$ \cite{BD79,BD80}.   
If a spherical $(2m-1)$-design $X$ has exactly $m$ inner products, that is,  
\[
|\{\langle x,y \rangle \colon\, x,y \in X, x\ne y\}| =m,
\]  
then $X$ is called an {\it $m$-sharp set}. 
An $m$-sharp set is a universally optimal code \cite{CK07}, which not only minimizes certain potential energies but also maximizes the minimal pairwise distance for a fixed number of points (see \cite{CK07} for a list of known examples). Moreover, it has the structure of an association scheme \cite{DGS77}. 

The concept of spherical designs has been generalized to {\it Euclidean designs}, which are weighted designs supported on multiple concentric spheres (see the comprehensive survey \cite{BBTZ17}). A natural lower bound on Euclidean designs is known, and a Euclidean design is said to be {\it tight} if it attains this bound.

Many examples of $m$-stiff configurations exhibit structures related to tight Euclidean or spherical designs, as well as $m$-sharp sets \cite{B22,BDHSSarXiv}.  
The following results illustrate the similarities between $m$-stiff configurations and tight Euclidean designs supported on two concentric spheres.  
\begin{itemize}
\item For an $m$-stiff configuration $X$, the $m$ elements of $D(z,X)$ coincide with the zeros of the degree-$m$ Gegenbauer polynomial \cite{B22,B24,BDHSS23, FL95}.  
\item For a tight Euclidean $(2m+1)$-design $X_1 \cup X_2$, where each $X_i$ lies on a sphere,  
the $m$ elements of $\{\langle x_1, x_2 \rangle \colon\, x_1 \in X_1, x_2 \in X_2\}$ coincide with the zeros of the degree-$m$ Gegenbauer polynomial \cite{BB14}.  
\end{itemize}
We are interested in the combinatorial relationships and structural similarities among Euclidean designs, $m$-stiff configurations, and the embeddings of tight spherical designs or $m$-sharp sets into a small number of parallel hyperplanes.

This paper focuses on the existence of $m$-stiff configurations. Specifically, our main results are as follows.  
\begin{enumerate}
    \item For each fixed integer $m > 5$, there exists no $m$-stiff configuration in $S^{d-1}$ for sufficiently large $d$. 
    \item For each fixed integer $d > 10$, there exists no $m$-stiff configuration in $S^{d-1}$ for sufficiently large $m$.
\end{enumerate}
Both results are established by analyzing the zeros of Gegenbauer polynomials.  
Result (1) essentially follows from a theorem in \cite{BB14}.  
The analysis of the zeros of Gegenbauer polynomials used in \cite{BD79,BD80} to establish the non-existence of tight spherical designs is also instrumental in considering Result (2) from a similar perspective.  

Furthermore, for small values of $d$ or $m$, the $m$-stiff configurations are classified as follows.
For $m = 2, 3, 4, 5$, we provide a complete classification of the dimensions in which $m$-stiff configurations exist.
For $m=6,7,8,9,10$, we prove that no $m$-stiff configuration exists in any dimension greater than $2$. In this case, since all dimensions in which an $m$-stiff configuration could exist can be listed explicitly according to Result (1), an exhaustive computer search becomes practicable.
For $d=3,4,\ldots,120$, we prove the non-existence of $m$-stiff configurations in $S^{d-1}$ for all degrees $m > 5$.
For small values of $d\le 10$, we employ a case-by-case argument, which includes the use of the Newton polygon method.
For $11 \leq d \leq 120$, Result (2) provides a practical upper bound on $m$ for which $m$-stiff configurations may exist, allowing for an exhaustive computer search. 

Based on the above existence results, we conjecture the non-existence of $m$-stiff configurations in $S^{d-1}$ for all $(d, m)$ with $d \geq 3$ and $m \geq 6$.

\section{Gegenbauer polynomials and Christoffel numbers}
This section states basic results on Christoffel numbers of Gegenbauer polynomials. The notation follows Szeg\"{o} \cite{Sbook}. 
Jacobi's polynomials $P_n^{(\alpha,\alpha)}(x)$ with $\alpha=(d-3)/2$ are called {\it Gegenbauer polynomials} of dimension $d$, where $P_n^{(\alpha,\alpha)}(x)$ are orthogonal polynomials with respect to the weight function $(1-x^2)^\alpha$ on $x \in [-1,1]$, normalizing $P_n^{(\alpha,\alpha)}(1)=\binom{n+\alpha}{n}$ \cite[Section 4.1]{Sbook}. 

Let $x_1< \cdots <x_n$ be the zeros of $P_n^{(\alpha,\alpha)}(x)$. There uniquely exist 
real numbers $\lambda_1,\ldots, \lambda_n$ such that 
for each polynomial $\rho(x)$ of degree at most $2n-1$, 
\[
\frac{1}{h_0}\int_{-1}^1 \rho(x) (1-x^2)^\alpha dx=
\lambda_1 \rho(x_1)+\cdots+ \lambda_n \rho(x_n),
\]
where $h_0=\int_{-1}^1(1-x^2)^{\alpha}dx$. 
The numbers $\lambda_i$ are called {\it Christoffel numbers} of $P_n^{(\alpha,\alpha)}(x)$ \cite[Theorem 3.4.1]{Sbook}. 
The Christoffel numbers depend only on the inner product of orthogonal polynomials and are independent of the normalization of $P_n^{(\alpha,\alpha)}(1)$; here we use $\langle f,g\rangle=1/h_0\int_{-1}^1 f(x) g(x) (1-x^2)^\alpha dx$. 
Writing $P_n(x)=P_n^{(\alpha,\alpha)}(x)$, the Christoffel numbers can be expressed as follows,  \cite[equation (3.4.3)]{Sbook}: 
\begin{equation} \label{eq:lam}
\lambda_\nu=\frac{1}{h_0}\int_{-1}^1 \frac{P_n(x)}{P_n'(x_\nu)(x-x_\nu)}(1-x^2)^\alpha dx
= \frac{1}{h_0}\int_{-1}^1 \prod_{i\ne \nu} \frac{(x-x_i)}{(x_\nu-x_i)}(1-x^2)^\alpha dx,     
\end{equation}
and  \cite[equations (3.4.8), (4.3.3), and (4.5.2)]{Sbook}: 
\begin{equation} \label{eq:lam^-1}
\lambda_\nu^{-1}=K_n(x_\nu,x_\nu)
=\sum_{i=0}^n\frac{h_0}{h_i} \left(P_i(x_\nu)\right)^2>0,     
\end{equation}
where 
\begin{equation}
h_i=\int_{-1}^1\left(P_i(x)\right)^2(1-x^2)^\alpha dx=\frac{2^{2\alpha+1}}{2i+2\alpha+1} \frac{\Gamma(i+\alpha+1)^2}{\Gamma(i+1)\Gamma(i+2\alpha+1)}.    
\end{equation}
Note that $h_0/h_i$ is rational, and hence $\lambda_\nu^{-1}=f(x_\nu)$ for some $f \in \mathbb{Q}[x]$, see \cite[equation~(4.21.2)]{Sbook} for the expression of $P_i(x)$. 

The following lemma is used later. 
\begin{lemma}\label{lem:1}
Let $d\geq 3$ and $\alpha=(d-3)/2$. 
Let \(\lambda_1, \ldots, \lambda_n\) be the Christoffel numbers of $P_n^{(\alpha,\alpha)}(x)$, and $x_1,\ldots,x_n$ the zeros.  
    If all Christoffel numbers \(\lambda_i\) are rational numbers, then for each \(i \in \{1, \ldots, n\}\), the degree of the minimal polynomial of \(x_i\) over \(\mathbb{Q}\) is at most 2 (i.e., \(x_i\) is either rational or quadratic irrational).
\end{lemma}

\begin{proof}
 Assume \(x_i\) is an irrational number of degree at least 3 (i.e., the degree of the minimal polynomial of $x_i$ is at least 3), and let \(x_j\) and \(x_k\) be two distinct conjugates of \(x_i\) over \(\mathbb{Q}\).
Consider the field \(K\) obtained by adjoining all conjugates of \(x_i\) to $\mathbb{Q}$, and take automorphisms \(\sigma\) and \(\tau\) of \(K\) over \(\mathbb{Q}\) such that
\[
\sigma(x_i) = x_j, \quad \tau(x_i) = x_k.
\]

From \eqref{eq:lam^-1}, 
there exists \(f \in \mathbb{Q}[x]\) such that \(f(x_\nu) = \lambda_\nu^{-1}\) for each $\nu \in \{1,\ldots, n\}$. 
If $\lambda_i,\lambda_j$ are rational numbers, then
\[
\lambda_i^{-1} = \sigma(\lambda_i^{-1}) = \sigma(f(x_i)) = f(\sigma(x_i)) = f(x_j) = \lambda_j^{-1}.
\]
Similarly, by using \(\tau\), we obtain
$ \lambda_i^{-1} = \lambda_k^{-1}$.

It is known that \(\lambda_\nu\) satisfy the following unimodal property:
\begin{align*}
    &\lambda_1 < \lambda_2 < \cdots < \lambda_{n/2}=\lambda_{n/2+1} >  \cdots > \lambda_{n-1} > \lambda_n &\text{if $n$ is even;}\\
    &\lambda_1 < \lambda_2 < \cdots < \lambda_{(n+1)/2-1}<\lambda_{(n+1)/2} > \lambda_{(n+1)/2+1} >  \cdots > \lambda_{n-1} > \lambda_n &\text{if $n$ is odd,}
\end{align*}
where the corresponding zeros satisfy $x_1<x_2<\cdots<x_n$. 
We provide a proof of this fact here.  
From \cite[equation (15.3.1)]{Sbook}, there exists a positive constant \( c \) depending only on \( n \) and \( \alpha \) such that, for each \( \nu \in \{1,\ldots, n\} \), the Christoffel numbers can be expressed as  
\[
\lambda_\nu^{-1} = c(1-x_\nu^2) \left(\frac{d}{dx} P_n^{(\alpha,\alpha)}(x_\nu) \right)^{2}. 
\]
Consider the polynomial function $F$ defined by  
\[
n(n+2\alpha +1) F(x) = n(n+2\alpha +1) \left(P_n^{(\alpha,\alpha)}(x)\right)^2 + (1-x^2) \left(\frac{d}{dx} P_n^{(\alpha,\alpha)}(x) \right)^{2}.
\]
This function $F(x)$ passes through all reciprocals of constant multiples of the Christoffel numbers, given by  
\[
\left(x_\nu,\frac{1}{cn(n+2\alpha+1)} \lambda_\nu^{-1} \right). 
\]
From \cite[equation (7.32.4)]{Sbook} and $d\ge 3$ (this means $2 \alpha +1 > 0$), we have  
\[
n(n+2\alpha+1) \frac{d}{dx} F(x) = 2(2\alpha+1)x \left( \frac{d}{dx} P_n^{(\alpha,\alpha)}(x) \right)^2,
\]
which implies that $F(x)$ is monotonically decreasing for $x < 0$ and monotonically increasing for $x > 0$.  
Therefore, it follows that $\lambda_\nu$ satisfy the unimodal property.  

By the unimodal property, at most two of the Christoffel numbers can be equal. However, in our case, we have found that three of them are equal.  
Thus, the degree of $x_i$ must be at most 2.  
\end{proof}

\section{Non-existence of $m$-stiff configurations} \label{sec:3}
In this section, we prove that for a given dimension, no $m$-stiff configuration exists for sufficiently large $m$. Borodachov \cite{B24} provided a necessary and sufficient condition for the existence of $m$-stiff configurations. This result plays an important role in the present paper as our starting point.
\begin{theorem}[{\cite[Theorem 6.8]{B24}}]
    \label{thm:exists_m-stiff}
    Suppose $m \geq 1$, $d\geq 2$. Let $x_1,\ldots, x_m$ be the zeros of $P_m^{(\alpha,\alpha)}(x)$ with $\alpha=(d-3)/2$. 
    Let $\varphi_k(x)=\prod_{i\ne k} (x-x_i)/(x_k-x_i)$, which satisfies $\varphi_k(x_\ell)=\delta_{k\ell}$. 
    Let 
    \[
    a_0(\varphi_k)=\frac{1}{h_0}\int_{-1}^1\varphi_k(x) (1-x^2)^\alpha dx. 
    \]
    There exists an $m$-stiff configuration of $S^{d-1}$ if and only if 
     $a_0(\varphi_k)$ is a positive rational number for each $k \in \{1,\ldots,m\}$. 
\end{theorem}
Even if the necessary and sufficient condition holds, this remains an existence theorem that depends on the existence of spherical designs and does not provide an explicit construction (see \cite{SZ84} for an existence theorem of spherical $t$-designs). 

For the case where $d=2$ or $m=1$, the $m$-stiff configurations can be easily classified.   
When $d=2$, the values $a_0(\varphi_k)$ are all positive rational numbers for any $m\geq 1$. Indeed, the regular $2m$-gon in $S^1$ is the $m$-stiff configuration. 
For $m = 1$, we have $a_0(\varphi_1) = 1$, and 
$X \subset S^{d-1}$ is a 1-stiff configuration 
if and only if it is a spherical 1-design
in some $(d-1)$-dimensional subspace of $\mathbb{R}^d$ (i.e., a set whose centroid is at the origin).

From \eqref{eq:lam}, the value $a_0(\varphi_k)$ coincides with
the Christoffel number $\lambda_k$. 
For $d\geq 3$, if $\lambda_k$ are all rationals, the zeros $x_i$ of $P_m^{(\alpha,\alpha)}(x)$ are rational or quadratic irrational from Lemma~\ref{lem:1}. In this case, noting that 
$P_m^{(\alpha,\alpha)}(x)$ is either an even or odd function, the square of each zero $x_i$ becomes a rational number. 
For the non-existence of $m$-stiff configurations, 
we prove that there exists a zero of $P_m^{(\alpha,\alpha)}(x)$ whose square is irrational. 

The following theorem is derived from a result in \cite{BB14}. 
\begin{theorem}
    For each fixed integer $m>5$, there exists no $m$-stiff configuration in $S^{d-1}$ for sufficiently large $d$. 
\end{theorem} 
\begin{proof}
Lemma 3.1 in \cite{BB14} states that
if the square of every zero of $P_m^{(\alpha,\alpha)}(x)$ is rational, then there exists a constant $C$, which depends only on $m$, such that $d<C$. This result implies the theorem.
\end{proof}
Unfortunately, the upper bound $C$ is not very effective.
Indeed, determining the integer solutions of many elliptic curves is required to settle the value of $C$. 

We employ a strategy similar to that in \cite{BB14,BD79,BD80} to establish the non-existence of $m$-stiff configurations for sufficiently large $m$. It is a method based on analyzing the zeros or coefficients of the polynomial $S_m(x)$ defined below.
Let $C_m^{d/2}(x)=c P_m^{(\alpha+1,\alpha+1)}(x)$ for $\alpha=(d-3)/2$, where $c$ is some constant.  
It is noteworthy that $C_m^{d/2}(x)$ is a Gegenbauer polynomial of dimension $d+2$. Define $S_m(x)$ as the monic polynomial whose zeros are the reciprocals of the non-zero zeros of $C_m^{d/2}(x)$, which does not depend on the constant $c$. 
For $X=x^2$ and $n=\lfloor m/2 \rfloor$, 
Bannai--Damerell \cite{BD80} provided the expression of the polynomial $S_m(X)$ as follows:  
\[
S_m(X)=X^n+\sum_{r=1}^n(-1)^r u_{r} X^{n-r}, 
\]
where 
\[
u_{r}=\binom{n}{r}
\frac{h(h+2)(h+4) \cdots (h+2r-2)}{1\cdot 3 \cdot 5 \cdots (2r-1)}, \qquad 
h=d+2n
\]
for $m=2n$, and
\[
u_{r}=\binom{n}{r}
\frac{h(h+2)(h+4) \cdots (h+2r-2)}{1\cdot 3 \cdot 5 \cdots (2r+1)}, \qquad 
h=d+2n+2
\]
for $m=2n+1$. 

Let ${\rm ord}_p(\gamma )$ be the highest power of a prime $p$ dividing $\gamma \in \mathbb{Z}\setminus \{0\}$ and ${\rm ord}_p(0)=\infty$. 
For a rational number $\gamma =a/b$ with $a,b \in \mathbb{Z}$, let ${\rm ord}_p(\gamma )={\rm ord}_p(a )-{\rm ord}_p(b)$. 
\begin{lemma} \label{lem:m=2n}
Suppose $m=2n$ with $n\geq 1$. 
    If a zero of $S_m(X)$ is rational, then 
    the zero is an integer. 
\end{lemma}
\begin{proof}
    Let $a$ be the least integer such that $au_i$ is an integer for each $i\in \{1,\ldots,n\}$. 
    Let $X=\beta/\gamma$ be a zero of $S_m(X)$, where $\beta$ and $\gamma$ are coprime integers. It follows that  
    \begin{equation} \label{eq:au_r}
    0=a \gamma^n S_m(\beta/\gamma)=a \beta^n+\sum_{r=1}^n(-1)^r au_{r} \gamma^r \beta^{n-r}. 
    \end{equation}
    Since $\gamma$ and $\beta$ are coprime, $a$ is divisible by $\gamma$. Assume that $\gamma$ is not equal to $1$. Then, there exists a prime divisor $p$ that divides both $a$ and $\gamma$. 
    From the minimality of $a$, there exists $i \in \{1,\ldots, n\}$ such that $au_i$ is not divisible by $p$. 

    For $p=2$, the denominator of $u_i$ is odd, and hence  ${\rm ord}_2 (au_i)\geq 1$ for each $i \in \{1,\ldots, n\}$, which is a contradiction. 

    For $p\geq 3$, one has 
    \begin{align*}
    {\rm ord}_p((2r- 1)!!)&= {\rm ord}_p\frac{(2r- 1)!}{2^r r!}
    =\sum_{i\geq 1} \left\lfloor \frac{2r-1}{p^i}\right\rfloor-\sum_{i\geq 1}\left\lfloor \frac{r}{p^i} \right\rfloor \\
    &\leq \sum_{i\geq 1}  \frac{2r-1}{p^i}=\frac{2r-1}{p-1}\leq \frac{2r-1}{2}    
    \end{align*}
    and hence ${\rm ord}_p((2r- 1)!!)\leq  r-1$ for $r\geq 1$. 
    This implies that ${\rm ord}_p(u_r)\geq -r+1$. 
    Therefore, ${\rm ord}_p(u_r\gamma^r)\geq (-r+1)+r=1$ for each $r\geq 1$. 
    If ${\rm ord}_p(a) =k$ holds, then ${\rm ord}_p(au_r\gamma^r)\geq k+1$ for each $r$. This implies that ${\rm ord}_p(a\beta^n)\geq k+1$ from \eqref{eq:au_r}. 
    However, since $\beta$ is coprime to $\gamma$ and not divisible by $p$, it follows that ${\rm ord}_p(a\beta^n)= k$, which is a contradiction. 

    Therefore, we can conclude $\gamma=1$. 
\end{proof}
\begin{lemma} \label{lem:m=2n+1}
    Suppose $m=2n+1$ with $n\geq 1$. 
    If a zero of $S_m(X)$ is rational, then 
    the zero can be expressed by $k/3$ for some $k\in \mathbb{Z}$. 
\end{lemma}
\begin{proof}
 Let $a$ be the least integer such that $au_i$ is an integer for each $i\in \{1,\ldots,n\}$. Let $X=\beta/\gamma$ be a zero of $S_m(X)$, where $\beta$ and $\gamma$ are coprime integers. 
 By the same manner as the proof of Lemma \ref{lem:m=2n}, if $\gamma\ne 3^\ell$ ($\ell \geq 0$), then there exists a prime divisor $p\ne 3$ that divides both $a$ and $\gamma$. From the minimality of $a$, there exists $i \in \{1,\ldots, n\}$ such that $au_i$ is not divisible by $p$.  
    
    For $p=2$, the denominator of $u_i$ is odd, and hence  ${\rm ord}_2 (au_i)\geq 1$ for each $i \in \{1,\ldots, n\}$, which is a contradiction. 

    For $p\geq 5$, one has 
    \begin{align*}
    {\rm ord}_p((2r+ 1)!!)&= {\rm ord}_p\frac{(2r+ 1)!}{2^r r!}=\sum_{i\geq 1} \left\lfloor \frac{2r+1}{p^i}\right\rfloor-\sum_{i\geq 1}\left\lfloor \frac{r}{p^i} \right\rfloor \\
    &\leq \sum_{i\geq 1}  \frac{2r+1}{p^i}\leq  \frac{2r+1}{p-1}\leq \frac{2r+1}{4},    
    \end{align*}
    and hence ${\rm ord}_p((2r+ 1)!!)\leq \lfloor (2r+1)/4) \rfloor\leq r-1$ for $r\geq 1$. 
    This implies that ${\rm ord}_p(u_r)\geq -r+1$. 
    In the same manner as the proof of Lemma \ref{lem:m=2n}, this leads to a contradiction. 

Therefore, $\gamma=3^\ell$ for some $\ell\geq 0$. 
We prove ${\rm ord}_3(\gamma) \leq 1$, namely $\ell=0,1$.  
Assume $\ell \geq 2$. 
As in the above calculation, for $p=3$, 
\[
{\rm ord}_3((2r+ 1)!!)\leq  \frac{2r+1}{p-1}=r+\frac{1}{2}, 
\]
and hence ${\rm ord}_3((2r+ 1)!!) \leq r$ for $r\geq 1$. 
This implies that ${\rm ord}_3(u_r)\geq -r$, and ${\rm ord}_3(u_r \gamma^r)\geq -r+2r=r\geq 1$. Thus, this leads to a contradiction. 
Therefore, $\ell=0,1$ holds. 
\end{proof}

By Lemma \ref{lem:m=2n}, for the case $m=2n$, 
if each zero of $C_m^{d/2}(x)$ is of degree at most 2, then all coefficients $u_r$ in $S_m(X)$ should be integers. 
By Lemma \ref{lem:m=2n+1}, for the case $m=2n+1$, 
if each zero of $C_m^{d/2}(x)$ is of degree at most 2, then all coefficients $u_r$ in $S_m(X)$ should satisfy $u_r =k/3^\ell$ with some $k,\ell \in \mathbb{Z}$. 
If these conclusions lead to a contradiction, then we establish the non-existence of an $m$-stiff configuration in $S^{d+1}$. 
The constant term of $S_m(X)$ is written as follows: 
For $m=2n$,  
\[
u_{n,d}^{-}=\frac{h(h+2)(h+4) \cdots (h+2n-2)}{1\cdot 3 \cdot 5 \cdots (2n-1)}  
\]
with $h=d+2n$. 
For $m=2n+1$, we define
\[
u_{n,d}^{+}=\frac{h(h+2)(h+4) \cdots (h+2n-2)}{1\cdot 3 \cdot 5 \cdots (2n+1)}
\]
with $h=d+2n+2$. It is noteworthy that 
$u_{n,d}^+=u_{n,d+2}^-/(2n+1)$. 
We prove that these coefficients $u_{n,d}^\pm$ are not 
integers for given $d$ and sufficiently large $m$. 

We prove the theorem below in the following subsections.
\begin{theorem}
    \begin{enumerate}
        \item For each fixed even integer $d\geq 8$, 
        there is no $2n$-stiff configuration in $S^{d-1}$ for sufficiently large $n$. 
        \item For each fixed even integer $d\geq 12$, 
        there is no $(2n+1)$-stiff configuration in $S^{d-1}$ for sufficiently large $n$. 
        \item For each fixed odd integer $d\geq 3$, 
        there is no $m$-stiff configuration in $S^{d-1}$ for sufficiently large $m$. 
    \end{enumerate}
\end{theorem}
\subsection{The case in which $d=2k$ even and $m=2n$ even}
\label{sec:d_even_m_even}
In this case, for given $d\geq 8$ ($k\geq 4$), we prove that the constant term $u_n=u_{n,d}^-$ of $S_m(x)$ is not an integer for sufficiently large $m$.  For this purpose, we provide a few lemmas. 
\begin{lemma} \label{lem:key}
    For any $n\in \mathbb{N}$, it follows that
    \[
    1\cdot 3 \cdot 5 \cdots (2n-1)2^{n-1}= n (n+1)  (n+2)\cdots (2n-1).
    \]    
\end{lemma}
\begin{proof}
We prove it by induction on $n$. 
    It is clear for $n=1$. 
    By the inductive assumption, 
    \begin{align*}
    1\cdot 3 \cdot 5 \cdots (2n+1)2^{n}&=1\cdot 3 \cdot 5 \cdots (2n-1)2^{n-1}\{2(2n+1)\}\\
    &=n (n+1)  (n+2)\cdots (2n-1) \{2 (2n+1)\}     \\
    &=(n+1)  (n+2)\cdots (2n-1) 2n (2n+1). \qquad \qedhere
    \end{align*}
\end{proof}

\begin{lemma}
    For $d=2k$ with $k\geq 2$, it follows that
    \begin{equation} \label{eq:u_n2}
    u_n=2^{2n+\lfloor (k-1)/2 \rfloor} \frac{(2n+1)(2n+3)\cdots (2n+ 2\lfloor k/2 \rfloor -1)}{(n+\lfloor (k-1)/2   \rfloor+1)(n+\lfloor (k-1)/2   \rfloor+2)\cdots (n+\lfloor (k-1)/2   \rfloor+\lfloor k/2 \rfloor)}
    \end{equation}
    for each $n \in \mathbb{N}$. 
\end{lemma}
\begin{proof}
    Fix $k$, and use induction on $n$.  
    Suppose $k = 2\ell$ with $\ell \geq 1$, and we prove that
    \[
    u_n = 2^{2n+\ell-1} \frac{(2n+1)(2n+3)\cdots (2n+2\ell -1)}{(n+\ell)(n+\ell+1)\cdots (n+2\ell -1)}. 
    \] 
    The case $k = 2\ell + 1$ can be handled in a similar manner.

    \medskip

    (i) By the definition of $u_n$, we have $u_1 = 2k + 2$.  
    From Lemma~\ref{lem:key}, the right-hand side of \eqref{eq:u_n2} with $n = 1$ is computed as follows:  
    \[
    2^{\ell +1} \frac{3 \cdot 5 \cdots (2\ell +1)}{(\ell +1)(\ell +2) \cdots 2\ell} 
    = 2 \cdot \frac{3 \cdot 5 \cdots (2\ell +1)}{1 \cdot 3 \cdot 5 \cdots (2\ell -1)} 
    = 2(2\ell +1) = 2k + 2.
    \]

    (ii) By the definition of $u_n$, it follows that
    \[
    u_n = \frac{4(n+\ell -1)(2n+2\ell -1)}{(2n-1)(n+2\ell -1)} \, u_{n-1}.
    \]
    By the inductive hypothesis for $n-1$, we have
    \begin{align*}
    u_n &= \frac{4(n+\ell -1)(2n+2\ell -1)}{(2n-1)(n+2\ell -1)} \cdot 2^{2n+\ell -3} \frac{(2n-1)(2n+1)\cdots (2n+2\ell -3)}{(n+\ell -1)(n+\ell)\cdots (n+2\ell -2)} \\
    &= 2^{2n+\ell -1} \frac{(2n+1)(2n+3)\cdots (2n+2\ell -1)}{(n+\ell)(n+\ell+1)\cdots (n+2\ell -1)}. \qquad \qedhere
    \end{align*}
\end{proof}

Now, we have the necessary tools to prove the non-existence theorem. 
\begin{theorem} \label{thm:deven}
Suppose $d=2k\geq 8$ $(k\geq 4)$.
Let 
\[
n_k=(2\theta-1)(2\theta-3)\cdots (2\theta- 2\lfloor k/2 \rfloor +1),
\]
where $\theta=\lfloor (k-1)/2   \rfloor+2$. 
Note that $n_k$ is positive. 
    If  $n>n_k$, then $u_n=u_{n,d}^-$ is not an integer. 
    In particular, for each fixed even integer $d\geq 10$, there does not exist a $2n$-stiff configuration in $S^{d-1}$ for $n>n_k$.   
\end{theorem}
\begin{proof}
    From $k\geq 4$, one has $\lfloor k/2   \rfloor \geq 2$, and the expression \eqref{eq:u_n2} for $u_n$ has factors $n+\lfloor (k-1)/2   \rfloor+1$ or $n+\lfloor (k-1)/2   \rfloor+2$ in the denominator. 
    Either $n+\lfloor (k-1)/2   \rfloor+1$ or $n+\lfloor (k-1)/2   \rfloor+2$ is an odd integer. 
    Let $n+\theta'$ be the odd one ($\theta' =\lfloor (k-1)/2   \rfloor+1$ or $\lfloor (k-1)/2   \rfloor+2$). We prove that $n+\theta'$ does not divide the odd factor $(2n+1)(2n+3)\cdots (2n+ 2\lfloor k/2 \rfloor -1)$ of the numerator in \eqref{eq:u_n2}. Indeed, 
    \begin{align*}
        (2n+1)(2n+3)\cdots &(2n+ 2\lfloor k/2 \rfloor -1)\\
        &\equiv 
        (-2\theta'+1)(-2\theta'+3)\cdots (-2\theta'+ 2\lfloor k/2 \rfloor -1) \pmod{n+\theta'}
    \end{align*}
    From our assumption $n>n_k$, the absolute value of the last expression is less than $n$ because $\theta' \leq \theta$. Since the expression is the product of odd numbers, it is not equal to $0$. Therefore, it is not congruent to $0$ mod $n+\theta'$, which is our desire. 
 \end{proof}
 For $d=6$ ($k=3$), we use the coefficient $u_{n-1}$ to prove the non-existence. 
 \begin{theorem} \label{thm:d=6,even}
     For $d=6$, 
     the coefficient $u_{n-1}$ of $S_{2n}(X)$ is not an integer for $n>30$. 
     In particular, there does not exist a $2n$-stiff configuration in $S^7$ for $n>30$. 
 \end{theorem}
 \begin{proof}
     From Lemma \ref{lem:key}, $u_{n-1}$ can be expressed by 
     \[
     u_{n-1}=2^{2n-1}\frac{n(2n-1)(2n+1)}{(n+1)(n+2)}.
     \]
     The theorem can be proved in the same manner as the proof of Theorem \ref{thm:deven} with the expression of $u_{n-1}$. 
 \end{proof}
 \subsection{The case in which $d$ is odd and $m=2n$ even}
 \label{sec:d_odd_m_even}
 The key in this case is the existence of a prime as described in the following lemma.
 \begin{lemma}[Hanson \cite{H73}] \label{lem:hanson}
     For $n>\ell$, the product of $\ell$ consecutive integers $n(n+1)\cdots (n+\ell-1)$ contains a prime divisor greater than $3\ell/2$ with the exceptions $3\cdot 4$, $8\cdot 9$, and $6\cdot 7 \cdot 8 \cdot 9 \cdot 10$. 
 \end{lemma}
 \begin{theorem} \label{thm:u_n_odd_d}
 Suppose $d$ is an odd positive integer.  
     If $n\geq 2d+5$, then $u_n=u_{n,d}^-$ is not an integer. 
     In particular, for each fixed odd integer $d\geq 3$, there does not exist a $2n$-stiff configuration in $S^{d-1}$ for $n \geq 2d+5$. 
 \end{theorem} 
 \begin{proof}
From Lemma \ref{lem:key}, one has 
\[
u_n=2^{n-1}\frac{(d+2n)(d+2n+2)\cdots (d+4n-2)}{n(n+1)\cdots (2n-1)}.
\]
From Lemma \ref{lem:hanson}, $n-1$ consecutive integers 
 $n(n+1)\cdots (2n-2)$ contains a prime divisor $p$ greater than $3(n-1)/2$ except for $n=3,6$. Thus, the denominator of $u_n$ contains a prime divisor $p> 3(n-1)/2$. 
 We prove that no factor in the numerator of $u_n$ is a multiple of $p$.  
 This is true since each factor $d+2n+2i$ is odd, $p\leq 2n-2 < d+2n$, and $3p> 9(n-1)/2\geq d+4n-2$ for $n \geq 2d+5$. 
 Therefore, $u_n$ is not an integer. 
\end{proof}
\subsection{The case in which $m=2n+1$ odd}
\label{sec:m_odd}
For this case, using $u_n = u_{n,d}^+ = u_{n,d+2}^- / (2n+1)$, we prove that when $u_n$ is reduced to its lowest terms, its denominator has a prime factor other than $3$. 
First, we prove the non-existence of $(2n+1)$-stiff configurations for even dimensions $d$ in a manner similar to Theorem \ref{thm:deven}. 

\begin{theorem}\label{thm:m=2n+1}
Suppose $d=2k\geq 14$ $(k\geq 7)$. 
    Let 
\[
n_k=(2\theta-1)(2\theta-3)\cdots (2\theta- 2\lfloor k/2 \rfloor +1),
\]
where $\theta=\lfloor k/2   \rfloor+4$. 
Note that $n_k$ is positive. 
    If $n > n_k$, then $u_n = u_{n,d}^+$ cannot be expressed in the form $s/3^{\ell}$ for some integers $s$ and $\ell$.  
    In particular,  for each fixed even integer $d\geq 16$, 
    there does not exist
a $(2n+1)$-stiff configuration in $S^{d-1}$ for $n>n_k$.  
\end{theorem}
\begin{proof}
    Since $u_n=u_{n,d}^+=u_{n,d+2}^-/(2n+1)$ holds, we prove 
    the denominator of $u_{n,d+2}^-$ has a prime divisor that is not $3$. 
    We have $d+2 =2k' \geq 16$ and $k'=k+1$ from our assumption. From $\lfloor k'/2 \rfloor \geq 4$,  the expression \eqref{eq:u_n2} for $u_n$ has factors $n+\lfloor (k'-1)/2 \rfloor+1$, $n+\lfloor (k'-1)/2 \rfloor+2$, $n+\lfloor (k'-1)/2 \rfloor+3$, and $n+\lfloor (k'-1)/2 \rfloor+4$ in the denominator.  One of them is neither even nor a multiple of 3. 
    Let $n+\theta'$ be such a number whose prime divisors are all different from $2$ and $3$. 
    In the same manner as the proof of Theorem \ref{thm:deven}, 
    the numerator in \eqref{eq:u_n2} is not divisible by $n+\theta'$ under our assumption on $n$. Therefore, when $u_n$ is reduced to its lowest terms, its denominator has a prime factor other than $3$.
\end{proof}
 For $d=12$ (resp.\ $d=10$), we use the coefficient $u_{n-1}$ (resp.\ $u_{n-2}$) of $S_{2n+1}(X)$ to prove the non-existence. 
 \begin{theorem} \label{d=6}
     For $d=12$ and $n>4152$, the coefficient $u_{n-1}$ of $S_{2n+1}(X)$ cannot be expressed in the form $s/3^\ell$ for any integers $s$ and $\ell$. In particular, there does not exist a $(2n+1)$-stiff configuration in $S^{13}$ for $n>4152$. 
 \end{theorem}
 \begin{proof}
     From Lemma \ref{lem:key}, $u_{n-1}$ can be expressed by 
     \[
     u_{n-1}=2^{2n+1}\frac{n(2n+1)(2n+3)(2n+5)}{(n+3)(n+4)(n+5)(n+6)}.
     \]
     The theorem can be proved in the same manner as the proof of Theorem \ref{thm:m=2n+1}, using the expression for $u_{n-1}$. 
 \end{proof}
 \begin{theorem} \label{d=5}
    For $d=10$ and $n>10390$, the coefficient $u_{n-2}$ of $S_{2n+1}(X)$ cannot be expressed in the form $s/3^\ell$ for any integers $s$ and $\ell$. In particular, there does not exist a $(2n+1)$-stiff configuration in $S^{11}$ for $n>10390$. 
\end{theorem}
\begin{proof}
    From Lemma \ref{lem:key}, $u_{n-2}$ can be expressed by 
    \[
    u_{n-2}=2^{2n-1}\frac{\frac{(n-1)n}{2}(2n-1)(2n+1)(2n+3)}{(n+2)(n+3)(n+4)(n+5)}.
    \]
    The theorem can be proved in the same manner as the proof of Theorem \ref{thm:m=2n+1}, using the expression for $u_{n-2}$. 
\end{proof}

Next, we prove the non-existence of $(2n+1)$-stiff configurations for odd dimensions $d$ in a manner similar to Theorem \ref{thm:u_n_odd_d}. 
 \begin{theorem} \label{thm:d,m_odd}
 Suppose $d$ is an odd positive integer.  
     If $n\geq 2d+9$, then $u_n=u_{n,d}^+$ cannot be expressed in the form $s/3^\ell$ for any integers $s$ and $\ell$. 
     In particular, 
for each fixed odd integer $d\geq 3$, there does not exist
a $(2n+1)$-stiff configuration in $S^{d-1}$ for $n\geq 2d+9$. 
 \end{theorem}
 \begin{proof}
     Since $u_n=u_{n,d}^+=u_{n,d+2}^-/(2n+1)$ holds, we prove 
    the denominator of $u_{n,d+2}^-$ has a prime divisor that is not $3$. 
    In the proof of Theorem \ref{thm:u_n_odd_d}, 
    for $n\geq 2d+9=2(d+2)+5$, the denominator of $u_{n,d+2}^-$ has a prime divisor $p$ greater than $3(n-1)/2>3$. 
    The divisor $p\ne 3$ cannot divide the numerator, which is our desire.  
 \end{proof} 
\section{Degrees $m$ for which $m$-stiff configurations exist for small $d$}

\subsection{Computer searchs for small $d$}
\label{sec:small_d_computer}
In Section \ref{sec:3}, we established a crude upper bound \( n_d \in \mathbb{N} \) such that no $m$-stiff configuration exists for any $ \lfloor m/2 \rfloor > n_d $. 
Here, $n_d$ can be explicitly expressed as a function of $d$ alone.
For $\lfloor m/2 \rfloor \leq n_d$, the rationality of the zeros of $S_m(X)$ can be verified using Mathematica \cite{Mathematica}.
Given a fixed dimension $d$, this in principle enables the classification of degrees $m$ for which an $m$-stiff configuration exists, though in practice, computational limitations may arise for large $d$. 
\begin{theorem} \label{thm:m_for_2d}
Suppose $d$ is even with $8 \leq d \leq 120$.   
     There exist $2n$-stiff configurations in $S^{d-1}$ if and only if $n=1$.  
\end{theorem}
\begin{proof}
Fix $d'=2k$ ($4 \leq k \leq 59$), which corresponds to the dimension $d=d'+2$ ($10 \leq d \leq 120$) of the sphere $S^{d-1}$. 
Let $n_k$ be the value defined in Theorem \ref{thm:deven}. 
From Theorem \ref{thm:deven}, it is sufficient to prove the non-existence of $2n$-stiff configurations for any each $n$ with $1<n\leq n_k$. 
For $n=1$, the regular cross-polytope is a $2$-stiff configuration for any dimension. 

We first suppose $n$ is odd. 
Let $\theta_k$ be defined to be the value depending only on $k$ as follows: 
\[
\theta_k=\begin{cases}
    \lfloor \frac{k-1}{2} \rfloor+1 \text{ if $k\equiv0,3 \pmod 4 $},\\
    \lfloor \frac{k-1}{2} \rfloor+2 \text{ if $k\equiv1,2 \pmod 4 $}.
\end{cases}
\]
Then, $n+\theta_k$ is an odd integer. 
Note that $n+\theta_k$ is a factor of the denominator of $u_n$ in the expression \eqref{eq:u_n2}. 
As explained in the proof of Theorem \ref{thm:deven}, if a $2n$-stiff configuration exists, then it satisfies  
\[
f(k) = (-2\theta_k+1)(-2\theta_k+3) \cdots (-2\theta_k+2 \lfloor k/2 \rfloor -1) \equiv 0 \pmod{n+\theta_k}.
\]  
Therefore, the possible values of $n$ are obtained by subtracting $\theta_k$ from the divisors of $f(k)$.  
The divisors are computed using the built-in Mathematica function \texttt{Divisors}.  
Let $H_k$ be the set of such odd integers $n \geq 2$.  

The odd part of an integer $s$ is given by $s' = s / 2^{{\rm ord}_2(s)}$.  
For each $n \in H_k$, we check whether the odd part of $u_n'$ is an integer by computing  
\begin{equation} \label{eq:mod}
\prod_{i=1}^{\lfloor k/2 \rfloor}(2n+2i-1) \mod \left( \prod_{i=1}^{\lfloor k/2 \rfloor} (n+\lfloor \frac{k-1}{2}\rfloor +i)\right)',  
\end{equation}  
which is expressed in terms of the odd parts of both the numerator and the denominator of \eqref{eq:u_n2}. 
If \eqref{eq:mod} is not congruent to 0, then $n$ is discarded.  
Otherwise, we check the integrality of $u_{n-1}'$ in a similar way.  
These two conditions eliminate almost all elements of $H_k$.  
For the remaining values of $n \in H_k$, we factorize $S_{2n}(X)$ using the Mathematica function \texttt{Factor}.  If the zeros of $S_{2n}(X)$ are all integers, then the Christoffel numbers $\lambda_k$ are all positive rational numbers (see \eqref{eq:lam^-1}), and a $2n$-stiff configuration exists by Theorem \ref{thm:exists_m-stiff}. 
Indeed, $S_{2n}(X)$ cannot be factorized using only linear polynomials for $n>1$.

The cases where $n$ is even or $d=8$ can be handled similarly.  
\end{proof}
The calculation in the proof of Theorem \ref{thm:m_for_2d} reduces the computational cost by removing the power of 2 from the expression \eqref{eq:u_n2} for $u_n$. 
From similar computations, we obtain the following theorems.
\begin{theorem}
Suppose $d$ is even with $12 \leq d \leq 120$.
Then, there exist $(2n+1)$-stiff configurations in $S^{d-1}$ if and only if $n=0,1$ or $(d,n)=(26,2)$.  
\end{theorem}
\begin{proof}
    The idea is similar to the proof of Theorem \ref{thm:m_for_2d}. Note that we should choose $\theta_k$ such that $n+\theta_k$ is neither even nor a multiple of $3$. 
\end{proof}
\begin{theorem}
   Suppose $d$ is odd with $3 \leq d \leq 1999$.   
     There exist $m$-stiff configurations in $S^{d-1}$ if and only if $m=1,2,3$ or $(d,m)=(23,4),(241,4),(241,5),(1079,5)$.  
\end{theorem}
\begin{proof}
    This case is easier since the lower bounds on $m$ given in Theorems~\ref{thm:u_n_odd_d} and \ref{thm:d,m_odd} are small. 
\end{proof}

\subsection{Remaining cases of $d\le 10$}
In this subsection, 
we discuss the existence of $m$-stiff configurations with small $d$ that is not covered in Subsection~\ref{sec:small_d_computer}.
Specifically, the remaining cases are $2n$-stiff configurations in $S^{d-1}$ for $d=4,6$ and $(2n+1)$-stiff configurations in $S^{d-1}$ for even $d$ with $4\le d \le 10$.
For these cases, we give proofs in order of similarity of the arguments as follows:
$(m,d)=(\text{even},4)$ in Theorem~\ref{thm:even_d=4},
$(m,d)=(\text{odd},4)$ in Theorem~\ref{thm:odd_d=4},
$(m,d)=(\text{odd},6)$ in Theorem~\ref{thm:odd_d=6},
$(m,d)=(\text{odd},8)$ in Theorem~\ref{thm:odd_d=8}, and
$(m,d)=(\text{odd},10)$ in Theorem~\ref{thm:odd_d=10}.
The case of $(m,d)=(\text{even},6)$ not yet mentioned requires the Newton polygon method, and will be described in Section~\ref{sec:Newton}.

In the following discussion, 
the existence of $m$-stiff configurations in any dimension for $m\le 3$ is used as a known fact. 
This will be discussed in more detail in Section~\ref{sec:small_m}.

\begin{theorem}
    \label{thm:even_d=4}
    There exists a $2n$-stiff configuration in $S^{3}$ if and only if $n=1$.
\end{theorem}
\begin{proof}
We prove the non-existence of $2n$-stiff configurations in $S^{3}$ for $n\ge 2$.
For $d=4$, we consider $u_n=u_{n,d'}^-$ for $d'=2$.
Then we have $u_n=2^{2n}$ for $n\ge 1$.
If $S_{2n}(X)$ has distinct $n$ integer zeros,
then each zero $X_i$ of $S_{2n}(X)$ forms $2^{l_i}$ for some non-negative integer $l_i$
because $S_{2n}(X)$ is monic and the constant term $u_n$ of $S_{2n}(X)$ is a power of $2$.
Hence, we have $2^{2n}=u_n=X_1 X_2 \cdots X_n$.
On the other hand, since $X_1, X_2, \ldots ,X_n$ are mutually distinct,
we have $X_1 X_2 \cdots X_n\ge 2^0 \cdot 2^1 \cdots 2^{n-1}=2^{\frac{n(n-1)}{2}}$.
Thus, we have $2n\ge \frac{n(n-1)}{2}$, i.e., $n\le 5$.
For $n=2,3,4,5$, one can check that $S_{2n}(X)$ does not have distinct integer zeros.
\end{proof}

\begin{theorem}
    \label{thm:odd_d=4}
    There exists a $(2n+1)$-stiff configuration in $S^{3}$ if and only if $n=0,1,2$.
\end{theorem}
\begin{proof}
    The existence of a $5$-stiff configuration in $S^{3}$ is referred to in Section~\ref{sec:5-stiff}.
    We prove the non-existence of $(2n+1)$-stiff configurations in $S^{3}$ for $n\ge 3$.
    For $d=4$, we consider $u_r$ for $d'=2$ and we have
    \begin{equation}
        \label{eq:d6odd-1}
        u_r=\binom{n}{r}\frac{(2n+4)(2n+6)\cdots (2n+2r+2)}{1\cdot 3 \cdots (2r+1)}=\frac{2^{2r}}{n+1}\binom{n+r+1}{2r+1}
    \end{equation}
    for $n\ge 1$.
    In particular, we have $u_n=2^{2n}/(n+1)$.
    Thus $u_n$ is expressed in the form $s/3^l$ for some integers $s$ and $l$ if and only if $n=2^a \cdot 3^b -1$ for some integers $a$ and $b$.
    Then we have
    \begin{equation}
        \label{eq:d6odd-2}
        u_n=\frac{2^{2^{a+1} 3^b -a -2}}{3^b}.
    \end{equation}
    If $S_{2n}(X)$ has distinct $n$ zeros which are expressed in the forms $k/3$ for $k\in \Z$,
    then each zero $X_i$ of $S_{2n}(X)$ is expressed in the form $2^{k_i}\cdot 3^{l_i}$ for $k_i\ge 0$ and $l_i\ge -1$
    because $S_{2n}(X)$ is monic and the constant term $u_n$ of $S_{2n}(X)$ is expressed in \eqref{eq:d6odd-2}.
    Furthermore, $l_i$ must be $0$ or $-1$.
    Indeed, if $l_i>0$, then the equation $(n+1)S_{2n}(X_i)=0$ leads to
    \[
    (n+1)(2^{k_i}\cdot 3^{l_i})^n +\sum^{n-1}_{r=1} (-1)^r 2^{2r}\binom{n+r+1}{2r+1}(2^{k_i}\cdot 3^{l_i})^{n-r} = -(-1)^n 2^{2n}
    \]
    by \eqref{eq:d6odd-1} and the left hand side is divided by $3$.
    This is a contradiction.
    Hence, we have $u_n=X_1 X_2 \cdots X_n$
    and the $b$ zeros of $S_{2n}(X)$ are expressed in the form $2^k/3$ and the $n-b$ zeros of $S_{2n}(X)$ are expressed in the form $2^{k'}$.
    On the other hand, since $X_1, X_2, \ldots ,X_n$ are mutually distinct,
    we have 
    \begin{align*}
        X_1 X_2 \cdots X_n
        &\ge
        \frac{2^0}{3}\cdot \frac{2^1}{3} \cdots \frac{2^{b-1}}{3} \cdot 2^0 \cdot 2^1 \cdots 2^{n-b-1}\\
        &=
        \frac{2^{\frac{b(b-1)}{2}+\frac{(2^{a} 3^b-b-1)(2^{a} 3^b-b-2)}{2}}}{3^b}.
    \end{align*}
    Thus, we have $2^{a+1} 3^b -a -2\ge \frac{b(b-1)}{2}+\frac{(2^{a} 3^b-b-1)(2^{a} 3^b-b-2)}{2}$.
    This inequality has only solutions $(a,b)=(0,0),(1,0),(0,1),(2,0),(1,1),(0,2)$
    and these correspond to $n=0,1,2,3,5,8$ respectively.
    For $n=3,5,8$, one can check that the zeros of $S_{2n}(X)$ do not satisfy the desired forms.
\end{proof}

\begin{theorem}
    \label{thm:odd_d=6}
    There exists a $(2n+1)$-stiff configuration in $S^{5}$ if and only if $n=0,1$.
\end{theorem}
\begin{proof}
    The theorem can be proved in the same manner as the proof of Theorem~\ref{thm:odd_d=4} with the expression 
    \[
         u_r=\frac{2^{2r}(n+r+2)}{(n+1)(n+2)}\binom{n+r+1}{2r+1}, 
     \]
     in particular, $u_n=2^{2n+1}/(n+2)$. 
\end{proof}
To prove Theorems~\ref{thm:odd_d=8} and \ref{thm:odd_d=10}, 
we provide the following lemma. 
\begin{lemma}
    \label{lem:rational_sum}
    Let $p$ be a prime number, $r\ge 1, n_1,n_2,m_1,m_2$ be integers
    with $n_1,m_1,m_2\neq 0$ and ${\rm gcd}(n_1,p)={\rm gcd}(m_2,p)=1$.
    Then the rational number obtained from the sum of the two rationals
    \[
     \frac{n_1}{p^r m_1}+\frac{n_2}{m_2}
    \]
    must have the prime factor $p$ in the denominator.
\end{lemma}
\begin{proof}
    The sum of the two rationals is expressed as
    \[
        \frac{n_1}{p^r m_1}+\frac{n_2}{m_2}
        =
        \frac{n_1 m_2 +  p^r n_2 m_1}{p^r m_1 m_2}
    \]
    and,
    by the assumption ${\rm gcd}(n_1,p)={\rm gcd}(m_2,p)=1$, we have
    $n_1 m_2 +  p^r n_2 m_1 \equiv n_1 m_2 \not\equiv 0 \pmod{p}$.
    Thus, the numerator is not divisible by $p$.
\end{proof}

\begin{theorem}
    \label{thm:odd_d=8}
    There exists a $(2n+1)$-stiff configuration in $S^{7}$ if and only if $n=0,1$.
\end{theorem}
\begin{proof}
    We prove the non-existence of $(2n+1)$-stiff configurations in $S^{7}$ for $n\ge 2$.
    For $d=8$, we consider $u_n=u_{n,d'}^+$ for $d'=6$.
    Then we have 
    \begin{equation}
        \label{eq:und8}
        u_n=2^{2n+1}\frac{2n+3}{(n+2)(n+3)}=-\frac{2^{2n+1}}{n+2}+\frac{2^{2n+1}\cdot 3}{n+3}
    \end{equation}
    for $n\ge 1$.
    By Lemma~\ref{lem:hanson}, $(n+2)(n+3)$ contains a prime divisor greater than $3$ when $n\neq 1,6$.
    Let $p\ge 5$ be such a prime divisor.
    Then $p$ can only be contained in either $n+2$ or $n+3$.
    Moreover, by Lemma~\ref{lem:rational_sum}, the equation \eqref{eq:und8} implies that the denominator of $u_n$ contains a prime divisor $p$,
    which is our desire.
    For $n=6$, one can check that $u_3=\frac{14080}{7}$.
\end{proof}

\begin{theorem}
    \label{thm:odd_d=10}
    There exists a $(2n+1)$-stiff configuration in $S^{9}$ if and only if $n=0,1$.
\end{theorem}
\begin{proof}
    We prove the non-existence of $(2n+1)$-stiff configurations in $S^{9}$ for $n\ge 2$.
    For $d=10$, we consider $u_n=u_{n,d'}^+$ for $d'=8$ ($k=4$).
    Then we have 
    \[
        u_n=2^{2n+2}\frac{2n+3}{(n+3)(n+4)}=-\frac{2^{2n+2}\cdot 3}{n+3}+\frac{2^{2n+2}\cdot 5}{n+4}
    \]
    for $n\ge 1$.
    By Lemma~\ref{lem:hanson}, $(n+3)(n+4)$ contains a prime divisor greater than $3$ when $n\neq 5$.
    Let $p\ge 5$ be such a maximal prime divisor.
    If $p>5$ or $p=5$ does not vanish in the denominator of $u_n$, the same argument as the proof of Theorem~\ref{thm:odd_d=8} can be applied.
    We consider the case that $p=5$ vanishes in the denominator of $u_n$.
    Then we have $n+3=2^{r_1} \cdot 3^{r_2}$ and $n+4=2^{r_3} \cdot 3^{r_4} \cdot 5$ by Lemma~\ref{lem:rational_sum}.
    For $n\ge 2$, this only occurs when (i) $n+3=2^{r_1}$ and $n+4=3^{r_4} \cdot 5$ ($r_4\ge 1$) or (ii) $n+3=3^{r_2}$ and $n+4=2^{r_3}\cdot 5$ ($r_3\ge 1$).
    In the case (i), by $r_4\ge 1$, we have $n+4 \equiv 0 \pmod{15}$.
    On the other hand, we have $n+4=n+3+1=2^{r_1}+1 \equiv 2^l+1 \not\equiv 0 \pmod{15}$, where $r_1\equiv l \pmod{4}$ with $l=0,1,2,3$.
    This is a contradiction.
    In the case (ii), for $r_3\ge 2$, we have $n+4 \equiv 0 \pmod{20}$.
    On the other hand, we have $n+4=n+3+1=3^{r_2}+1 \equiv 3^l+1 \not\equiv 0 \pmod{20}$, where $r_2\equiv l \pmod{4}$ with $l=0,1,2,3$.
    This is a contradiction.
    For $r_3=1$, i.e., $n+4=10$.
    Then we have $n+3=9$ and $n=6$.
    In this case, we have $u_3=\frac{18304}{7}$. 

    For $n=5$, we have $u_3=\frac{7040}{5}$.
\end{proof}

\section{Dimensions $d$ where an $m$-stiff configuration exists for small $m$}
\label{sec:small_m}
In this section, we provide the complete list of
the dimensions $d\ge 2$ where an $m$-stiff configuration in $S^{d-1}$ exists for $m=2,3,4,5$.
We give the Gegenbauer polynomials $P^{(\alpha,\alpha)}_m(x)$ for $\alpha=\frac{d-3}{2}$
of small degree and their zeros for later use:
$P^{(\alpha,\alpha)}_0(x)=1$;
$P^{(\alpha,\alpha)}_1(x)=\frac{d-1}{2}x$ and its zero $x=0$;
$P^{(\alpha,\alpha)}_2(x)=\frac{d+1}{8}(d x^2-1)$ and its zeros $x=\pm \frac{1}{\sqrt{d}}$;
$P^{(\alpha,\alpha)}_3(x)=\frac{(d+3)(d+1)}{48}x((d+2)x^2-3)$ and its zeros $x=0,\pm \sqrt{\frac{3}{d+2}}$;
\[
    P^{(\alpha,\alpha)}_4(x)=\frac{(d+5)(d+3)}{384}((d+2)(d+4)x^4-6(d+2)x^2+3)
\]
and its zeros
\begin{equation}
    \label{eq:P4roots}
    x=\pm \sqrt{\frac{3(d+2)\pm \sqrt{6(d+1)(d+2)}}{(d+2)(d+4)}};
\end{equation}
\[
    P^{(\alpha,\alpha)}_5(x)=\frac{(d+7)(d+5)(d+3)}{3840}x((d+4)(d+6)x^4-10(d+4)x^2+15)
\]
and its zeros
\begin{equation}
    \label{eq:P5roots}
    x=0,\pm\sqrt{\frac{5(d+4)\pm \sqrt{10(d+1)(d+4)}}{(d+4)(d+6)}}.
\end{equation}

\subsection{Degree $m=2$}
It can be verified from a simple calculation that
the Christoffel numbers of $P^{(\alpha,\alpha)}_2(x)$ are
$\lambda_1=\lambda_2=\frac{1}{2}$.
By Theorem~\ref{thm:exists_m-stiff},
this implies there exists a $2$-stiff configuration in $S^{d-1}$ for each $d$.
In fact, the regular cross-polytopes and the cubes in $S^{d-1}$
for $d\ge 2$ and the demicube in $S^{d-1}$ for $d\ge 5$ are $2$-stiff configurations.
There are also several other examples of 2-stiff configurations
(cf.~Borodachov~\cite{B22, B24}).

\subsection{Degree $m=3$}
It can be verified from a simple calculation that
the Christoffel numbers of $P^{(\alpha,\alpha)}_3(x)$ are
$\lambda_1=\lambda_3=\frac{d+2}{6d}$ and $\lambda_2=\frac{2(d-1)}{3d}$.
By Theorem~\ref{thm:exists_m-stiff},
this implies there exists a $3$-stiff configuration in $S^{d-1}$ for each $d$.
There are some examples of $3$-stiff configurations
such as
the 24-cell in $S^3$,
Symmetrized Schl\"affi in $S^5$ and
the minimal shell of the $E_7$-root lattice in $S^7$,
etc.
(cf.~Borodachov~\cite{B22, B24}).
Note that explicit examples of 3-stiff configurations for all $d$ are not yet known.

\subsection{Degree $m=4$}
We calculate \eqref{eq:lam} using \eqref{eq:P4roots} and obtain that
the Christoffel numbers of $P^{(\alpha,\alpha)}_4(x)$ are
\begin{align*}
    \lambda_1=\lambda_4=&\frac{3d(d+1)-(d-2)\sqrt{6(d+1)(d+2)}}{12d(d+1)},\\
    \lambda_2=\lambda_3=&\frac{3d(d+1)+(d-2)\sqrt{6(d+1)(d+2)}}{12d(d+1)}.        
\end{align*}
By Theorem~\ref{thm:exists_m-stiff},
this implies there exists a $4$-stiff configuration in $S^{d-1}$ 
if and only if $d=2$ or $6(d+1)(d+2)$ is a square.
We determine all pairs $(d,y)$ of integers satisfying 
\begin{equation}
    \label{eq:m4square}
    6(d+1)(d+2)=y^2.   
\end{equation}
Since \eqref{eq:m4square} can be formed $(6d+9)^2-6y^2=9$,
we treat the generalized Pell equation
\begin{equation}
    \label{eq:m4Pell}
    x^2-6y^2=9,
\end{equation}
where $x=6d+9$.
Several methods for solving generalized Pell equations are known.
Here, we use the brute-force search.

Let $D$ be a nonsquare positive integer. For our purpose, we assume that $D \equiv 2 \pmod{4}$.
Let $K = \mathbb{Q}(\sqrt{D})$ be the quadratic number field generated by $\sqrt{D}$.
The norm in $K$ is given by $N(x + y\sqrt{D}) = x^2 - Dy^2$,
and the ring of integers of $K$ is $\mathcal{O}_K = \mathbb{Z}[\sqrt{D}]$.
Then, it is known that the group of units $\mathcal{O}_K^\ast:=\{\alpha \in \mathcal{O}_K \colon\, N(\alpha)=\pm 1\}$
has the unique element $u_1\in \mathcal{O}_K^\ast$ such that $\mathcal{O}_K^\ast =\{\pm 1\}\times \{u_1^k \colon\, k\in \Z\}$ and $u_1>1$,
which is called the fundamental unit for $\Q(\sqrt{D})$.
Let $M$ be a nonzero integer.
Then a pair $(x,y)$ of integers is a solution of $x^2-Dy^2=M$ if and only if $\alpha=x+y\sqrt{D}\in \mathcal{O}_K$ satisfies $N(\alpha)=M$.
We say that $\alpha ,\beta \in K$ are associates
(or are associated with one another)
in $\mathcal{O}_K^\ast$ if there exists $u\in \mathcal{O}_K^\ast$ such that $\alpha = u \beta$.

\begin{lemma}[cf.~{\cite[Section~8.8]{AG76}}]
\label{lem:brute-force}
Fix $u_0=a+b\sqrt{D}$ with positive integers $a,b$
and $N(u_0)=1$.
Every solution of $x^2-Dy^2=M$ is $(x'+y'\sqrt{D})u_0^k$
where $k\in \Z$, $N(x'+y'\sqrt{D})=M$ and
\[
|x'|\le \frac{u_1+|M|}{2}
\quad
\text{and}
\quad
|y'|\le \frac{u_1+|M|}{2\sqrt{D}},
\]
where $u_1$ is the fundamental unit for $\Q(\sqrt{D})$.
\end{lemma}

Returning to \eqref{eq:m4Pell},
it is known that the fundamental unit for $\Q(\sqrt{6})$
is $u_1=5+2\sqrt{6}$ and we apply Lemma~\ref{lem:brute-force}
as $u_0=u_1$ and $M=9$.
Then the candidates are
\[
|x'|\le \frac{5+2\sqrt{6}+9}{2}<10
\quad
\text{and}
\quad
|y'|\le \frac{5+2\sqrt{6}+9}{2\sqrt{6}}<4.
\]
Since $9+6y'^2=x'^2$ is square, $y'$ determines only $y'=0$.
Hence $x'=\pm 3$ and these are associates in $\mathcal{O}_K^\ast$.
Therefore the solution of \eqref{eq:m4Pell} are $u_\ell=3(5+2\sqrt{6})^\ell$ for $\ell\in \Z$.
In particular, we have $x=(u_\ell+ \overline{u_\ell})/2=3((5+2\sqrt{6})^\ell+(5-2\sqrt{6})^\ell)/2$.
Then 
\[
    d= \frac{x-9}{6}= \frac{(5+2\sqrt{6})^\ell+(5-2\sqrt{6})^\ell-6}{4}
\]
is a positive integer (i.e., a solution of \eqref{eq:m4square}) if and only if $\ell \ge 1$.
We give a list of solutions of \eqref{eq:m4square} for $d$ less than $10^8$ in Table~\ref{table:d4}.

\begin{table}[h]
    \caption{List of $d$, the zeros of  $P^{(\alpha,\alpha)}_4(x)$ and the Christoffel numbers for 4-stiff configurations}
\begin{center}
    \begin{tabular}{lll}
    \hline
    $d$ & zeros of  $P^{(\alpha,\alpha)}_4(x)$ & $\lambda_1=\lambda_4$, $\lambda_2=\lambda_3$ \\ \hline \\[-2ex]
    2 & $\pm \frac{1}{\sqrt{2-\sqrt{2}}}$, $\pm \frac{1}{\sqrt{2+\sqrt{2}}}$ & $\frac{1}{4}$, $\frac{1}{4}$\\[1ex]
    23 & $\pm \frac{1}{\sqrt{5}}$, $\pm \frac{1}{\sqrt{45}}$ & $\frac{11}{184}$, $\frac{81}{184}$\\[1ex]
    241 & $\pm \frac{1}{\sqrt{45}}$, $\pm \frac{1}{21}$ & $\frac{125}{2651}$, $\frac{2401}{5302}$\\[1ex]
    2399 & $\pm \frac{1}{21}$, $\pm \frac{1}{\sqrt{4361}}$ & $\frac{8829}{191920}$, $\frac{87131}{191920}$\\[1ex]
    23761 & $\pm \frac{1}{\sqrt{4361}}$, $\pm \frac{1}{\sqrt{43165}}$ & $\frac{237699}{5179898}$, $\frac{1176125}{2589949}$\\[1ex]
    235223 & $\pm \frac{1}{\sqrt{43165}}$, $\pm \frac{1}{\sqrt{427285}}$ & $\frac{8546759}{186296616}$, $\frac{84601549}{186296616}$\\[1ex]
    2328481 & $\pm \frac{1}{\sqrt{427285}}$, $\pm \frac{1}{\sqrt{4229681}}$ & $\frac{115260250}{2512430999}$, $\frac{2281910499}{5024861998}$\\[1ex]
    23049599 & $\pm \frac{1}{\sqrt{4229681}}$, $\pm \frac{1}{\sqrt{41869521}}$ & $\frac{8290175641}{180708856160}$, $\frac{82064252439}{180708856160}$\\[1ex]
    \hline
 \end{tabular}
\end{center}
\label{table:d4}
\end{table}

Some examples of $4$-stiff configurations are known.
For $d=2$, $X$ is a $4$-stiff configuration if and only if $X$ is a regular $8$-gon.
For $d=23$, a tight spherical $7$-design on $S^{22}$ is a $4$-stiff configuration.
No other explicit examples of $4$-stiff configurations seem to be known.

\subsection{Degree $m=5$}
\label{sec:5-stiff}
In the same way as for the case $m=4$,
the dimensions $d$ for $m=5$ can also be determined.
We calculate \eqref{eq:lam} using \eqref{eq:P4roots} and obtain that
the Christoffel numbers of $P^{(\alpha,\alpha)}_5(x)$ are
\begin{align*}
    \lambda_1=\lambda_5&=\frac{(d+1)(d+4)(7d+2)-(d-2)(2d+7)\sqrt{10(d+1)(d+4)}}{60d(d+1)(d+2)},\\
    \lambda_2=\lambda_4&=\frac{(d+1)(d+4)(7d+2)+(d-2)(2d+7)\sqrt{10(d+1)(d+4)}}{60d(d+1)(d+2)},\\
    \lambda_3&=\frac{8(d+1)(d-1)}{15d(d+2)}.
\end{align*}
By Theorem~\ref{thm:exists_m-stiff},
this implies there exists a $5$-stiff configuration in $S^{d-1}$ 
if and only if $d=2$ or $10(d+1)(d+4)$ is a square.
We determine all pairs $(d,y)$ of integers satisfying 
\begin{equation}
    \label{eq:m5square}
    10(d+1)(d+4)=y^2.   
\end{equation}
Since \eqref{eq:m5square} implies that $y$ is divisible by $5$ and \eqref{eq:m5square} can be formed $(10d+25)^2-10y^2=225$,
we treat the generalized Pell equation
\begin{equation}
    \label{eq:m5Pell}
    \tilde{x}^2-10\tilde{y}^2=9,
\end{equation}
where $\tilde{x}=2d+5$ and $\tilde{y}=y/5$.
It is known that the fundamental unit for $\Q(\sqrt{10})$
is $u_1=3+\sqrt{10}$ and $N(u_1)=-1$.
We apply Lemma~\ref{lem:brute-force}
as $u_0=u_1^2=19+6\sqrt{10}$ and $M=9$.
Then the candidates are
\[
|x'|\le \frac{3+\sqrt{10}+9}{2}<8
\quad
\text{and}
\quad
|y'|\le \frac{3+\sqrt{10}+9}{2\sqrt{10}}<3.
\]
Since $9+10y'^2=x'^2$ is square, $y'$ determines only $y'=0,\pm 2$.
If $y'=0$, then $x'=\pm 3$ and these are associates in $\mathcal{O}_K^\ast$.
If $y'=\pm 2$, then $x'=\pm 7$.
Obviously, the pairs $\pm(7+2\sqrt{10})$ and $\pm(7-2\sqrt{10})$ are associates in $\mathcal{O}_K^\ast$, respectively.
Moreover, one can check easily that $3$, $7+ 2\sqrt{10}$ and $7- 2\sqrt{10}$ are not associates each other.
Therefore, the solution of \eqref{eq:m5Pell} are
$u_{\ell_1}=3(19+6\sqrt{10})^{\ell_1}$,
$u_{\ell_2}=(7+2\sqrt{10})(19+6\sqrt{10})^{\ell_2}$ and
$u_{\ell_3}=(7-2\sqrt{10})(19+6\sqrt{10})^{\ell_3}$ and
for $\ell_1,\ell_2,\ell_3\in \Z$.
In particular, we have $\tilde{x}=(u_{\ell_i}+ \overline{u_{\ell_i}})/2$
and
\begin{align*}
    d= \frac{\tilde{x}-5}{2}
    =&\frac{3((19+6\sqrt{10})^{\ell_1}+(19-6\sqrt{10})^{\ell_1})-10}{4},\\
    &\frac{(7+2\sqrt{10})(19+6\sqrt{10})^{\ell_2}+(7-2\sqrt{10})(19-6\sqrt{10})^{\ell_2}-10}{4},\\
    &\frac{(7-2\sqrt{10})(19+6\sqrt{10})^{\ell_3}+(7+2\sqrt{10})(19-6\sqrt{10})^{\ell_3}-10}{4}
\end{align*}
are positive integers (i.e., solutions of \eqref{eq:m5square}) if and only if $\ell_1 \ge 1$, $\ell_2 \ge 0$, $\ell_3 \ge 0$, respectively.
We give a list of solutions of \eqref{eq:m5square} for $d$ less than $10^8$ in Table~\ref{table:d5}.

\begin{table}[h]
\begin{center}
\caption{List of $d$, the zeros of  $P^{(\alpha,\alpha)}_5(x)$ and the Christoffel numbers for 5-stiff configurations}\label{table:d5}
    \begin{tabular}{lll}
    \hline
    $d$ & zeros of  $P^{(\alpha,\alpha)}_5(x)$ & $\lambda_1=\lambda_5$, $\lambda_2=\lambda_4$, $\lambda_3$ \\ \hline \\[-2ex]
    2& $\pm \sqrt{\frac{5+\sqrt{5}}{8}}, \pm \sqrt{\frac{5-\sqrt{5}}{8}},0$ & $\frac{1}{5}, \frac{1}{5}, \frac{1}{5}$\\[1ex]
    4&$\pm \frac{\sqrt{3}}{2},\pm \frac{1}{2}, 0$& $\frac{1}{12}, \frac{1}{4}, \frac{1}{3}$\\[1ex]
    26&$\pm \frac{1}{2},\pm \frac{1}{4},0$& $\frac{5}{273}, \frac{64}{273}, \frac{45}{91}$\\[1ex]
    124& $\pm \frac{1}{4},\pm \sqrt{\frac{3}{208}},0$ & $\frac{41}{3255}, \frac{2197}{9765}, \frac{1025}{1953}$\\[1ex]
    241& $\pm \sqrt{\frac{3}{91}},\pm \frac{1}{\sqrt{133}}, 0$ & $\frac{30976}{58563}, \frac{15379}{1288386}, \frac{48013}{214731}$\\[1ex]
    1079& $\pm \frac{1}{\sqrt{133}},\pm \frac{1}{\sqrt{589}}, 0$ & $\frac{620928}{1166399}, \frac{319333}{27993576}, \frac{6226319}{27993576}$\\[1ex]
    4801& $\pm \frac{1}{\sqrt{589}},\pm \sqrt{\frac{3}{7843}}, 0$ & $\frac{12293120}{23059203}, \frac{4252580}{376633649}, \frac{502022587}{2259801894}$\\[1ex]
    9244& $\pm \sqrt{\frac{3}{3400}},\pm \frac{1}{\sqrt{5032}}, 0$ & $\frac{1898923}{3561251}, \frac{3453125}{306267586}, \frac{68026979}{306267586}$\\[1ex]
    41066& $\pm \frac{1}{\sqrt{5032}},\pm \frac{1}{\sqrt{22348}}, 0$ & $\frac{112427757}{210812311}, \frac{277761368}{24665040387}, \frac{5477735041}{24665040387}$\\[1ex]
    182404& $\pm \frac{1}{\sqrt{22348}},\pm \sqrt{\frac{3}{297772}}, 0$ & $\frac{739360427}{1386316001}, \frac{11924172077}{1059145424764}, \frac{235212857191}{1059145424764}$\\[1ex]
    351121& $\pm \sqrt{\frac{3}{129055}},\pm \frac{1}{\sqrt{191065}}, 0$ & $\frac{65752510208}{123286658883}, \frac{581549060605}{51657110071977}, \frac{7647903391205}{34438073381318}$\\[1ex]
    1559519& $\pm \frac{1}{\sqrt{191065}},\pm \frac{1}{\sqrt{848617}}, 0$ & $\frac{1297119739392}{2432102630399}, \frac{24970041242125}{2218077598923888}, \frac{492582157057067}{2218077598923888}$\\[1ex]
    6926641& $\pm \frac{1}{\sqrt{848617}},\pm \sqrt{\frac{3}{11307439}}, 0$ & $\frac{25588456289536}{47978369396163}, \frac{670100637133543}{59525163630839562}, \frac{19828663190016109}{89287745446259343}$\\[1ex]
    13333444& $\pm \sqrt{\frac{3}{4900636}},\pm \frac{1}{\sqrt{7255420}}, 0$ & $\frac{11852048593409}{22222594446003}, \frac{1634104921847879}{145157986921291596}, \frac{10745365944241375}{48385995640430532}$\\[1ex]
    59220746& $\pm \frac{1}{\sqrt{7255420}},\pm \frac{1}{\sqrt{32225080}}, 0$ & $\frac{233806450453101}{438387109404751}, \frac{21926672426094515}{1947753927085308693}, \frac{432549261434995960}{1947753927085308693}$\\[1ex]    
    \hline
 \end{tabular}
\end{center}
\end{table}

Note that the only known explicit example of a 5-stiff configuration is a regular $10$-gon in $S^1$.

\section{Newton polygon method}
\label{sec:Newton}
Using the Newton polygon method, we prove the non-existence of $m$-stiff configurations in $S^{d-1}$ for small values of $m$ or $d$.  

For a polynomial with integer coefficients, the Newton polygon method \cite{D06} provides a necessary condition on the coefficients for all zeros of the polynomial to be integers.  
We recall the method here.  
Let $F(x) = \sum_{i=0}^k a_i x^{k-i}$ be a polynomial with integer coefficients $a_i$, and define $c_i(p) = {\rm ord}_p(a_i)$ for a prime number $p$.  
Plot the points $(0,0), (1,c_0(p)),(2,c_1(p)),\ldots, (k+1,c_k(p))$ in the Cartesian plane, and take the lower convex hull of these points, which is called the Newton polygon.  
If all the zeros of $F(x)$ are integers, then all the slopes of the Newton polygon must be integers for each prime number $p$. 

\subsection{Non-existence for $m=6,7,8,9,10$}
For $m=6,7,8,9,10$, we prove there exists no $m$-stiff configuration in any dimensions. 

Bannai and Bannai \cite{BB14} provided a necessary condition for all zeros of $S_m(X)$ to be rational numbers by using the Newton polygon method.  
We outline a sketch of their argument here.  
Assume that the zeros of $S_m(X)$ are rational and write them as $p_i/q_i$ ($i=1,\ldots, n$), where $p_i, q_i \in \mathbb{Z}$ and $n=\lfloor m/2 \rfloor$.  
Define $y_i$ as  
\[
y_i={\rm lcm}(q) \frac{p_i}{q_i}, \qquad \text{where ${\rm lcm}(q)={\rm lcm}\{q_1,\ldots, q_n\}$}.  
\]
Then, $y_1, \ldots, y_n$ are the zeros of a certain polynomial $f(x)=\sum_{i=0}^n a_i x^{n-i}$ with integer coefficients $a_i$; see \cite[equations (5.10), (5.11), (5.12)]{BB14} for the precise expressions of $a_i$.  

We apply the Newton polygon method to the polynomial $f(x)$ for $m\geq 6$ (equivalently, $n\geq 3$).  
Let $h=d+2n+1+(-1)^{m-1}$ as previously defined.  
For a prime number $p$ such that $p>2n+(-1)^{m-1}$ and $h+2(n-2)$ is divisible by $p$, we can show that  
\[
c_0(p) = c_1(p) = \cdots = c_{n-2}(p) = 0,  
\]
and  
\[
c_{n-1}(p) = c_n(p) = {\rm ord}_p(h+2(n-2)) > 0.  
\]
Similarly, for a prime number $p'$ such that $p'>2n+(-1)^{m-1}$ and $h+2(n-3)$ is divisible by $p'$, we can show that  
\[
c_0(p') = c_1(p') = \cdots = c_{n-3}(p') = 0,  
\]
and  
\[
c_{n-2}(p') = c_{n-1}(p') = c_n(p') = {\rm ord}_{p'}(h+2(n-3)) > 0.  
\]
Since the slopes of the Newton polygon must be integers,  
$c_{n}(p)$ is even and $c_{n}(p')$ is a multiple of $3$.  

We express  
\[
h+2(n-2) = Ay^2 \quad \text{and} \quad h+2(n-3) = Bx^3,
\]
where $A, B, x$, and $y$ are positive integers, $A$ is a product of distinct prime numbers, and $B$ is a product of prime numbers, each raised to an exponent of at most $2$.  
From the above argument, any prime factor of $A$ and $B$ is at most $2n+(-1)^{m-1}$, and hence they have only finitely many possible values depending only on $m$.  

For given integers $A$ and $B$, we consider the Diophantine equation  
\begin{equation}\label{eq:Ay^2-Bx^3=2}
Ay^2 - Bx^3 = 2,
\end{equation}
which has only finitely many integer solutions \cite{B69}; see \cite[Chapter 28]{Mbook}.  
For each solution $(x,y)$, we obtain a candidate for the dimension as  
\[
d = Ay^2 - 4n + 3 + (-1)^{m-1}.
\]

For $m = 6,7,8,9,10$, we solve the possible Diophantine equations \eqref{eq:Ay^2-Bx^3=2} via computer computations using Magma \cite{Magma} (version 2.28-19).  
Each prime number at most $2n+(-1)^{m-1}$ is at most $7$ for $m \leq 10$.  
Therefore, we select $A$ and $B$ from the sets  
\begin{equation} \label{eq:list_of_AB}
A \in \{2^{r_1} 3^{r_2} 5^{r_3} 7^{r_4} \colon r_1,r_2,r_3,r_4 \in \{0,1\}\}, \  
B \in \{2^{r_1} 3^{r_2} 5^{r_3} 7^{r_4} \colon r_1,r_2,r_3,r_4 \in \{0,1,2\}\}.    
\end{equation}

Applying the variable transformation $Y = A^2 B y$ and $X = ABx$ to \eqref{eq:Ay^2-Bx^3=2}, we obtain a Weierstrass standard form  
\begin{equation} \label{eq:satandard_form}
Y^2 = X^3 + 2 A^3 B^2.     
\end{equation}  
This type of elliptic equation can be solved if a set of generators of the Mordell-Weil group can be computed \cite{ST94}.  
Using the built-in Magma function \texttt{MordellWeilGroup}, we can compute a complete set of generators of the Mordell-Weil group for \eqref{eq:satandard_form} for each pair of $A$ and $B$ from \eqref{eq:list_of_AB}.  
For some larger values of $A$ and $B$, this built-in function fails to compute the complete set of generators.  
Indeed, such cases occur when $m = 11$.  

We compute all integer solutions of \eqref{eq:satandard_form} using the Magma function \texttt{IntegralPoints} with the option \texttt{SafetyFactor := 10}, which expands the search range and reinforces the accuracy of the results.  

For each candidate pair $(d,m)$ obtained from the above computation, we verify that $S_m(X)$ cannot be factored into a product of linear polynomials.  
Based on these computational results, we state the following theorem.  

\begin{theorem}
    If $d > 2$ and $m = 6,7,8,9,10$,  
    then no $m$-stiff configuration exists in $S^{d-1}$.
\end{theorem}

\subsection{Non-existence for $d=6$ and $m=2n$ even}
\begin{theorem}
    There exists a $2n$-stiff configuration in $S^{5}$ if and only if $n=1$.
\end{theorem}
\begin{proof}
We prove the non-existence of $2n$-stiff configurations in $S^{5}$ for $n\ge 2$.
For $d=6$, we consider the constant term $u_n=u_{n,d'}^-$ of $S_{2n}(X)$ for $d'=4$.
From equation \eqref{eq:u_n2} we have
\[
u_n=2^{2n}\frac{2n+1}{n+1}=2^{2n}\left(2-\frac{1}{n+1}\right)
\]
for $n\ge 1$.
This implies that $u_n$ is an integer if and only if $n=2^l-1$ for $l\ge 1$.
When $n = 2^l - 1$ for $l \ge 2$, and the coefficients $u_r$ ($1 \le r \le n$) are integers, we show that the Newton polygon of $S_{2n}(X)$ with respect to $p = 2$ has at least one non-integer slope. 
Let $c_{i}(p) = {\rm ord}_p(u_{i})$ for $1\le i \le n$ and $c_0(p)={\rm ord}_p(1)=0$.
By 
$u_1=n(2n+4)=2(2^l-1)(2^l+1)$,
we have $c_{1}(2)=1$,
Similarly, 
we have $c_{2}(2)=3$ and $c_{3}(2)=4$.
Moreover, by
\[
u_r = \binom{n}{r}\frac{(2n+4)(2n+6)\cdots (2n+2r+2)}{1\cdot 3\cdots (2r-1)}
=2^r \binom{n}{r}\frac{(n+2)(n+3)\cdots (n+r+1)}{1\cdot 3\cdots (2r-1)}
\]
and the product of $r$ consecutive integers $(n+2)(n+3)\cdots (n+r+1)$ contains at least $\lfloor \frac{r}{2} \rfloor$ even integers,
we have $c_{r}(2)\ge r+\lfloor \frac{r}{2} \rfloor$.
This implies that the points $(r+1,c_{r}(2))$ are not below the line $y=\frac{3}{2}x-2$ for $r\ge 1$.
Then, the Newton polygon for $(0,0),(1,c_0(2)),\ldots ,(n+1,c_n(2))$
has the edge between $(2,1)$ and $(4,4)$ whose slope is $3/2$ (see Figure~\ref{fig:Newton_polygon}).
Therefore, there exists a non-integer zero of $S_{2n}(X)$.
\begin{figure}[htbp]
    \centering
\begin{tikzpicture}[xscale = 0.7, yscale = 0.7]
    \draw[thick, ->] (0,0) -- (5,0) node [below] {$x=r$};
    \draw[thick, ->] (0,0) -- (0,5) node [left] {$y=c_{r-1}(2)$};
    \node at (0,0) [anchor=north east] {O};
    \foreach \x in {1,2,3,4}
        \draw (\x ,0) node[anchor=north] {$\x$};
    \foreach \y in {1,2,3,4}
        \draw (0,\y) node[anchor=east] {$\y$};
    \draw [help lines] (0,0) grid (5,5);
    \coordinate (O) at (0,0);
    \coordinate (u0) at (1,0);
    \coordinate (u1) at (2,1);
    \coordinate (u2) at (3,3);
    \coordinate (u3) at (4,4);
    \coordinate (u4) at (4.666,5);
    \fill (O) circle (2pt);
    \fill (u0) circle (2pt);
    \fill (u1) circle (2pt);
    \fill (u2) circle (2pt); 
    \fill (u3) circle (2pt); 
    \draw (O)--(u0)--(u1)--(u3);
    \draw[dashed] (u3)--(u4);
\end{tikzpicture}
\caption{The Newton polygon for $S_{2n}(X)$ with $d'=4$ and $p=2$}
\label{fig:Newton_polygon}
\end{figure}
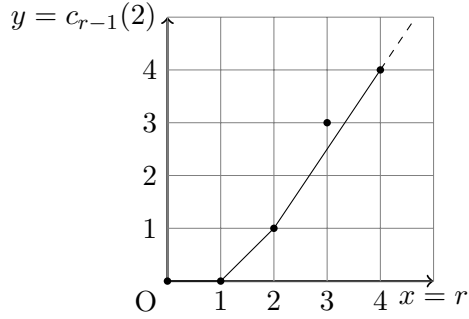
\end{proof}

\section{Concluding remarks}
The present paper has primarily established the non-existence of $m$-stiff configurations for sufficiently large $m$.  
The key idea was the analysis of the zeros of the Gegenbauer polynomial, inspired by the method in \cite{BD79, BD80}, which aimed at proving the non-existence of tight spherical $t$-designs.  
If we fix a small dimension $d$, we can, in principle, classify the $m$-stiff configurations.  
Based on the results concerning the existence of $m$-stiff configurations presented in this paper, we propose the following conjecture. 
\begin{conjecture}
Let $d$ be an integer greater than $2$.  
    There exists an $m$-stiff configurations in $S^{d-1}$ if and only if $m=2,3$, $(d,m)=(d_4,4)$ with  
    \[
    d_4= \frac{(5+2\sqrt{6})^\ell+(5-2\sqrt{6})^\ell-6}{4} 
    \]
    for $\ell \in \mathbb{N}$ with $\ell>1$, or $(d,m)=(d_5,5)$ with
    \begin{align*}
        d_5
    =&\frac{3((19+6\sqrt{10})^{\ell_1}+(19-6\sqrt{10})^{\ell_1})-10}{4},\\
    &\frac{(7+2\sqrt{10})(19+6\sqrt{10})^{\ell_2}+(7-2\sqrt{10})(19-6\sqrt{10})^{\ell_2}-10}{4},\\
    &\frac{(7-2\sqrt{10})(19+6\sqrt{10})^{\ell_3}+(7+2\sqrt{10})(19-6\sqrt{10})^{\ell_3}-10}{4}
    \end{align*}
    for $\ell_1, \ell_2,\ell_3 \in \mathbb{N}$. 
\end{conjecture}
Bannai and Damerell \cite{BD80} showed the non-existence of tight spherical $(2m-1)$-designs in $S^{d-1}$ for any $m \geq 5$ and $d \geq 3$, except for $(m,d) = (6,24)$, mainly using the Newton polygon method applied to the polynomial $S_m(X)$.
In their setting, the zeros of $S_m(X)$ are required to be square integers, which is a stronger condition than ours.
Although it might be possible to prove our conjecture by applying the Newton polygon method, the situation is not straightforward.

We are interested in the structural relationship among three objects: spherical $t$-designs, Euclidean $t$-designs, and $m$-stiff configurations.  
As an illustrative example, we consider the following case.  
The tight spherical 7-design $X$ on $S^{22}$ is a $4$-stiff configuration~\cite{BDHSSarXiv}. 
This was essentially known earlier in \cite{BBS12}, although the 
terminology 4-stiff was not used explicitly. 
Indeed, the set $X$ can be partitioned into four subsets  
\[
X = X_1 \cup X_2 \cup X_3 \cup X_4,
\]
such that there exists a point $x_0 \in S^{22}$ satisfying  
\[
\langle x_0 , x \rangle = -\frac{1}{\sqrt{5}}, -\frac{1}{\sqrt{45}}, \frac{1}{\sqrt{45}}, \frac{1}{\sqrt{5}}
\]
for $x \in X_1, X_2, X_3, X_4$, respectively, where $|X_1| = |X_4| = 275$ and $|X_2| = |X_3| = 2025$.  
The sets $X_i$ are embedded into the sphere $S^{21}$, and we denote their images by $X_i'$.  
Then, for some positive scalars $r_1, r_2 > 0$ and appropriate weights, the union $r_1 X_1' \cup r_2 X_2'$ forms a tight Euclidean $6$-design in $\mathbb{R}^{22}$ \cite{BBS12}. 
The sets $X_1'$ and $X_2'$ are both spherical 4-designs (with $X_1'$ being tight), and the numbers of distances between distinct points are 2 and 3, respectively.  
Under certain assumptions, if we can identify such a situation, the non-existence of tight Euclidean $t$-designs or $m$-stiff configurations may allow us to infer the non-existence of tight spherical $t$-designs whose existence is currently unknown.

\bigskip

\noindent
\textbf{Acknowledgments.} 
H. Kurihara was partially supported by JSPS KAKENHI Grant Number 24K06830. 
H. Nozaki was partially supported by JSPS KAKENHI Grant Numbers 22K03402 and 24K06688. 
We are grateful to the 
referees for their comments.

\end{document}